\documentclass[11pt]{amsart}
\usepackage{amsfonts,amssymb,amscd,amsmath,enumerate,verbatim,calc,graphicx}
\def\NZQ{\mathbb}               

\def\ZZ{{\NZQ Z}}
\def\RR{{\NZQ R}}
\def\NN{{\NZQ N}}
\def\PP{{\NZQ P}}
\def\kk{\Bbbk}%

\def\mm{\mathfrak{m}} 
\def\sat{{\rm sat}}
\def\Proj{\operatorname{Proj}}
\def\Tor{\operatorname{Tor}}
\def\codim{\operatorname{codim}}
\def\reg{\operatorname{reg}}

\def\cL{{\mathcal L}}

\def\opn#1#2{\def#1{\operatorname{#2}}} 
\opn\con{conv} \opn\inte{int} \opn\height{ht} \opn\vol{vol} \opn\gp{gp}
\newtheorem{theorem}{Theorem}[section]
\newtheorem{prop}[theorem]{Proposition}

\newtheorem{lem}[theorem]{Lemma}
\theoremstyle{remark}
\newtheorem{rem}[theorem]{Remark}
\newtheorem{dfn}[theorem]{Definition}
\numberwithin{equation}{section}

\begin{document}

\title[Non-level semi-standard graded CM domain]{Non-level semi-standard graded Cohen--Macaulay domain with $h$-vector $(h_0,h_1,h_2)$}
\author{Akihiro Higashitani}
\address{Department of Mathematics, Kyoto Sangyo University, 
Motoyama, Kamigamo, Kita-Ku, Kyoto, Japan, 603-8555}
\email{ahigashi@cc.kyoto-su.ac.jp}
\author{Kohji Yanagawa}
\address{Department of Mathematics, Kansai University, Suita, Osaka 564-8680, Japan}
\email{yanagawa@kansai-u.ac.jp}
\thanks{
The authors are partially supported by JSPS Grant-in-Aid for Young Scientists (B) 26800015, and Grant-in-Aid for Scientific Research (C) 16K05114, respectively.}
\begin{abstract}
Let $\kk$ be an algebraically closed field of characteristic 0, and $A=\bigoplus_{i \in \NN} A_i$ a Cohen--Macaulay graded domain with $A_0=\kk$. If $A$ is semi-standard graded (i.e., $A$ is finitely generated as  a $\kk[A_1]$-module), it has the {\it $h$-vector} $(h_0, h_1, \ldots, h_s)$, which encodes the Hilbert function of $A$.  
From now on, assume that $s=2$.  It is known that if  $A$ is standard graded (i.e., $A=\kk[A_1]$), then $A$ is level. We will show that, in the semi-standard case, if $A$ is not level, then $h_1+1$ divides $h_2$. Conversely, for any positive integers $h$ and $n$, there is a non-level $A$ with the $h$-vector $(1, h, (h+1)n)$. Moreover, such examples can be constructed as Ehrhart rings  (equivalently, normal toric rings). 
\end{abstract}

\maketitle

\section{Introduction}
Let $\kk$ be  a field, and $A=\bigoplus_{i \in \NN} A_i$ a graded noetherian commutative ring with $A_0=\kk$.  
If $A =\kk[A_1]$, that is, $A$ is generated by $A_1$ as a $\kk$-algebra,  we say $A$ is {\it standard graded}.  
If $A$ is finitely generated as a $\kk[A_1]$-module, we say $A$ is {\it semi-standard graded}.  
The {\it Ehrhart rings}  of lattice polytopes (see \S4 below) and the face rings of simplicial posets (see \cite{St2}) are typical examples of  semi-standard graded rings. 
In this sense, the notion of semi-standard graded rings is natural in combinatorial commutative algebra.   

If $A$ is a semi-standard graded ring of dimension $d$, its Hilbert series is of the form  
$$\sum_{i \in \NN} (\dim_\kk A_i) t^i =\frac{h_0+h_1 t + \cdots +h_st^s}{(1-t)^d}$$ 
for some integers $h_0, h_1, \ldots, h_s$ with $\sum_{i=1}^s h_i \ne 0$ and $h_s \ne 0$.  
We call the vector $(h_0,h_1,\ldots,h_s)$ the {\it $h$-vector} of $A$. 
We always have $h_0=1$ and $\deg A=\sum_{i=0}^s h_i$. 

If a semi-standard graded ring $A$ is Cohen--Macaulay,  its $h$-vector satisfies 
$h_i \ge 0$ for all $i$. If further $A$ is  standard graded, we have $h_i > 0$ for all $i$. 
For further information on the $h$-vectors of Cohen--Macaulay semi-standard (resp. standard) graded {\it domains}, see \cite{St} (resp. \cite{Y}).  

If a semi-standard graded ring $A$ is Cohen--Macaulay and of dimension $d$, it admits the (graded) canonical module $\omega_A$, 
which is a $d$-dimensional Cohen--Macaulay $A$-module with the Hilbert series 
$$\sum_{i \in \ZZ} \dim_\kk (\omega_A)_i t^i =\frac{t^{d-s}(h_s+h_{s-1} t + \cdots +h_0t^s)}{(1-t)^d},$$ 
where  $(h_0, h_1, \ldots, h_s)$ is the $h$-vector of $A$  (see \cite[\S 3.6 and  Theorem~4.4.6]{BH}). 

\begin{dfn}
In the above situation, if $\omega_A$ is generated by $(\omega_A)_{d-s}$ as an $A$-module (equivalently, $\omega_A$ is generated by elements all of the same degree),  we say $A$ is {\it level}. 
\end{dfn}

Clearly, the notion of level rings generalizes that of Gorenstein rings.  It is easy to see that a Cohen--Macaulay semi-standard graded ring with the $h$-vector $(h_0,h_1)$ or $(h_0, 0, h_2)$ is always level.
The following fact, whose assumption that $A$ is a domain is really necessary, is a special case of \cite[Theorem~3.5]{Y}. 
This essentially follows from the theory of algebraic curves, and might be an old  result.

\begin{prop}[{\cite[Corollary~3.11]{Y}}]
Let $\kk$ be an algebraically closed field of characteristic 0, and $A$ a Cohen--Macaulay standard graded $\kk$-algebra.  If the $h$-vector of $A$ is of the form  $(h_0, h_1, h_2)$ and $A$ is an integral domain, then $A$ is level. 
\end{prop}

In this paper, we weaken the assumption on $A$ in the above result to be semi-standard graded.  The following is the first main result.

\begin{theorem}[see Theorem~\ref{main1}]\label{main0} Let $\kk$ be an algebraically closed field of characteristic 0, and $A$ a Cohen--Macaulay semi-standard graded domain with the $h$-vector $(h_0,h_1,h_2)$.  If $A$ is not level, then ($h_1 \ne 0$ and)  $h_1+1$ divides $h_2$.  
\end{theorem} 
 
The outline of the proof is the following. By Bertini's theorem, we may assume that $\dim A=2$. Under the assumption of the theorem, if $A$ is not level, then the subring $B:=\kk[A_1]$ is isomorphic to the Veronese subring  $\kk[x^n, x^{n-1}y, \ldots, xy^{n-1}, y^n]$ of $\kk[x,y]$. Next we regard $A$ as a $B$-module. 
Then it is a maximal Cohen--Macaulay module, and we consider its direct sum decomposition.   
However, the classification of  indecomposable  maximal Cohen--Macaulay modules over $\kk[x^n, x^{n-1}y, \ldots, xy^{n-1}, y^n] \ (\cong B)$ is well-known, and we can determine the $B$-module structure of $A$.  

The next result states that the ``converse" of the above theorem holds. 

\begin{theorem}[see Theorem~\ref{thm:Ehr}] 
For any positive integers $h$ and $n$, there is a Cohen--Macaulay semi-standard graded domain which is not level and has the $h$-vector $(1, h, (h+1)n)$. Moreover, these rings can be constructed as Ehrhart rings  (equivalently, as normal affine semigroup rings).  
\end{theorem}

Hence, even if we restrict our attention to Ehrhart rings, Theorem~\ref{main0} has much sense. However we have no combinatorial proof of this theorem in the Ehrhart ring case. 

\section{Preliminaries}
In this section, we collect basic facts we will use in the next section. 
See \cite{BH} for undefined terminology and basic properties of Cohen--Macaulay rings.  
Throughout this section, let $A$ be  a semi-standard graded ring with $d= \dim A$, and $(h_0, h_1, \ldots, h_s)$ its $h$-vector. 
Note that $A$ is a graded local ring with the graded maximal ideal $\mm= \bigoplus_{i > 0} A_i$. 
Since $\dim A = \dim \kk[A_1]$, the ideal of $A$ generated by $A_1$ is $\mm$-primary. Hence, if $|\kk|=\infty$,  we can take a system of parameter $\theta_1, \ldots, \theta_d$ of $A$ from $A_1$.

Let $A^\sat$ denote the {\it saturation} $A/H_\mm^0(A)$ of $A$. It is clear that $A$ and $A^\sat$ define the same projective scheme, that is, we have $\Proj A = \Proj (A^\sat)$. 
In this paper, the {\it homogeneous coordinate ring} of a projective scheme $X \subset \PP^n$ means the {\it standard} graded ring $R$ with $X=\Proj R$ and $R=R^\sat$. Of course, there is a standard graded polynomial ring  $S=\kk[x_0, \ldots, x_n]$ with the graded surjection $f: S \to R$ which induces the inclusion map $\Proj R=X \hookrightarrow \PP^n =\Proj S$. We say $X \subset \PP^n$ is {\it non-degenerate}, if no hyperplane of $\PP^n$ contains $X$, equivalently, $\dim_\kk R_1= n+1$.  

In the rest of this section, we use the following convention.  
\begin{itemize}
\item $B =\kk[A_1]$ is the subalgebra of $A$ generated by $A_1$ over $\kk$.

\item $S=\kk[x_1, \ldots, x_m]$ is a standard  graded polynomial ring, where $m=\dim_\kk A_1=h_1+d$.  Note that $B$ can be seen as a quotient ring of $S$. 
\end{itemize}

Clearly, $A$ is a finitely generated graded $S$-module.  
For a finitely generated graded $S$-module $M$, $\beta_{i,j}^S(M)$ (or just   $\beta_{i,j}(M)$) denotes the graded Betti number $\dim_\kk [\Tor_i^S(\kk,M)]_j$ of $M$. We also set $\beta_i(M) :=\sum_{j \in \ZZ} \beta_{i,j}(M)$. 
Since $A_1 =S_1$, we have $\beta_{i,i}(A)=0$ for all $i >1$. 
If $A$ is Cohen--Macaulay, we have  
\begin{equation}\label{Betti dual}
\beta_{i,j}(\omega_A)=\beta_{m-d-i, m-j}(A),
\end{equation}
where $\omega_A$ is the canonical module of $A$. 
Let $r(A)$ denote the number of minimal generators of $\omega_A$ as a graded $A$-module, and call it the {\it Cohen--Macaulay type} of $A$. Clearly, $A$ is level if and only if $h_s =r(A)$.

For a finitely generated graded $S$-module $M$, 
$$\reg_S(M) :=\max \{ j-i \mid \beta_{i,j}(M) \ne 0 \}$$
is called the {\it Catelnuovo-Mumford regularity} of $M$ (\cite{EG}). 
While the theory of  Catelnuovo-Mumford regularities is very deep, we only use  elementary properties.
For example, if $A$ is Cohen--Macaulay, then we have $\reg A=s$.

The following easy result might be well-known to the experts, but we give the proof for the reader's convenience. 

\begin{lem}\label{h1=0}
If a Cohen--Macaulay semi-standard graded ring $A$ has the $h$-vector of the form $(h_0, 0, h_2)$, then $A$ is level.  
\end{lem}

\begin{proof}
We may assume that $|\kk| = \infty$. So we can take a system of parameter $\{\theta_1, \ldots, \theta_d\} \subset A_1$. Then  $A/(\theta_1, \ldots, \theta_d)$ has the same $h$-vector as $A$,  and $A$ is level if and only if so is $A/(\theta_1, \ldots, \theta_d)$. Hence we may assume that $\dim A=0$. In this case, we have  
$$\dim_\kk (\omega_A)_i = \begin{cases}
h_s & \text{if $i=-2$,}\\
1 \,(=h_0)   & \text{if $i=0$,}\\
0 & \text{otherwise.}
\end{cases}$$
If $A$ is not level, then $\mm \cdot (\omega_A)_{-2}=0$ and $\omega_A = (\omega_A)_{-2} \oplus  (\omega_A)_0$ as an $A$-module. This is a contradiction, since $\omega_A$
is indecomposable in general. 
\end{proof}
 
\begin{lem}\label{level by Betti}
Let $A$ be a Cohen--Macaulay semi-standard graded ring with the $h$-vector $(h_0, h_1, h_2)$.  Assume that $c:=h_1(=m-d)>0$.  Then $A$ is level if and only if  $\beta_{c, c+1}(A)=0$. 
\end{lem}

\begin{proof}
Since $\reg A =2$, we have $\beta_{c,j}(A)=0$ for all $j \ne c+1, c+2$. Hence if 
$\beta_{c, c+1}(A)=0$, then $\beta_0(\omega_A) = \beta_{0,d-2}(\omega_A)$ by \eqref{Betti dual}. 
It means that $\omega_A$ is generated by $(\omega_A)_{d-2}$ as an $S$-module, but it clearly implies that   $\omega_A$ is generated by $(\omega_A)_{d-2}$ as an $A$-module. Hence $A$ is level. 

Next we assume that $\beta_{c,c+1}(A) \ne 0$. Then $\beta_{0, d-1}(\omega_A) \ne 0$, and hence $S_1 \cdot (\omega_A)_{d-2} \subsetneq (\omega_A)_{d-1}$. 
Since $S_1 =A_1$, we have $A_1 \cdot  (\omega_A)_{d-2} \subsetneq (\omega_A)_{d-1}$, and $A$ is not level.  
\end{proof}

\begin{lem}\label{CM type}
If  a Cohen--Macaulay semi-standard graded ring $A$ has  the $h$-vector $(h_0, h_1, h_2)$ with $c:=h_1>0$,  then we have $r(A)=\beta_c(A)$.  
\end{lem}

\begin{proof}
As we have seen in the proof of Lemma~\ref{level by Betti}, we have $\beta_c(A)=\beta_0(\omega_A)= \beta_{0,d-2}(\omega_A)+ \beta_{0,d-1}(\omega_A)$. Since $A_1=S_1$, the number of minimal generators of $\omega_A$ as a graded $A$-module is equal to that as a graded $S$-module. 
\end{proof}

Assume that   $\kk$  is an algebraically closed field. 
It is a classical result that if $B$ is a domain (but not necessarily Cohen--Macaulay) then we have $$\deg B \ge \codim B+1,$$ where $\codim B:= \dim_\kk B_1-\dim B = m-d=h_1$. 
If the equality holds, then $B$ is Cohen--Macaulay. Moreover, Del Pezzo--Bertini's theorem gives a classification of standard graded domains $B$ with  $\deg B = \codim B+1$ (see, for example, \cite[Theorem~4.3]{EG}). 
In particular, if $\dim B=2$, then $B$ is the homogeneous coordinate ring of a rational normal curve.  

\section{$h_1+1$ divides $h_2$}
In this section, we always assume that the  base field $\kk$  is an algebraically closed field of characteristic 0. 
Let $X$ be  a finite set of points in the projective space $\PP^n=\Proj 
(\kk[x_0, \ldots, x_n])$, and $R$ its homogeneous coordinate ring. Then $R$ is  a Cohen--Macaulay standard graded ring with $\dim R=1$ and $\deg R = \# X$. We define the function $H_X:\ZZ \to \NN$ by $H_X(i)=\dim_\kk R_i$.  
If  $(h_0, h_1, \ldots, h_s)$ is the $h$-vector of $R$, we have $s=\min \{ i \mid H_X(i) = \# X\}$ and $h_i= H_X(i) -H_X(i-1)$ for all $0 \le i \le s$. 

\begin{dfn}[c.f. \cite{ACGH}]
Let $X \subset \PP^n$ be a finite set of points.  
We say that $X$ is in {\it uniform position}, if $H_X(1)=n+1$ (i.e., $X$  is non-degenerate) and every subset $Y \subset X$ satisfies $H_Y(i) = \min \{ H_X(i), \# Y \}$ for all $i$.
\end{dfn}

The usual definition of the uniform position property does not assume that $H_X(1)=n+1$, while many important examples satisfy it. Here we use the above definition for a quick exposition. 

Note that if $X \subset \PP^n$ is a finite set of points in uniform position, and $Y \subset X$ is a subset with $\# Y \ge n+1$, then $Y \subset \PP^n$ is in uniform position again.

The following  fundamental result is due to J. Harris.  An application of this result to commutative algebra is found in the paper \cite{Y} of the second author. 

\begin{theorem}[Uniform Position Theorem, {\cite[P. 113]{ACGH}}]\label{UPL}
If $C \subset \PP^n$ is a reduced, irreducible and non-degenerate curve, 
then a general hyperplane section $C \cap H$ is a set of points in uniform position in $H \cong \PP^{n-1}$. 
\end{theorem}

The following lemma must be well-known to the specialists, and there are several proofs. 
We will give one of them for the reader's convenience.

\begin{lem}\label{uniform level}
Let $X \subset \PP^n$ be a finite set of points in uniform position,  and 
$S=\kk[x_0, \ldots, x_n]$ (resp. $R$) the homogeneous coordinate rings of $\PP^n$ (resp. $X$). 
If $\deg R =\#X > n+1$, then $\beta_{n,n+1}^S(R)=0$. 
\end{lem}

\begin{proof}
Let $(h_0, h_1, \ldots, h_s)$ be the $h$-vector of $R$. Since $h_0=1$, $h_1 =n$ and $\# X = \deg R= \sum_{i=0}^s h_i$, 
we have $s \ge 2$.  We prove the assertion by induction on $\deg R$. 
If $\# X=n+2$, then it is easy to see that 
$$
H_X(i)=\begin{cases}
1 & \text{if $i=0$}, \\ 
n+1 & \text{if $i=1$}, \\
n+2 & \text{if $i \ge 2$}.\\  
\end{cases}
$$
So \cite[Corollary~2.5]{Kr} implies that $R$ is a Gorenstein (note that the Cayley-Bacharach property is weaker than the uniform position property).   Hence we have $\beta_n(R)=\beta_{n,n+2}(R)=1$  and $\beta_{n,n+1}(R)=0$.  

Next assume that $\# X  >  n+2$. Set $X':= X \setminus \{ p\}$ for a point $p \in X$, 
and let $R'$ be the homogeneous coordinate ring of $X'$. Consider the exact sequence 
\begin{equation}\label{IRR'}
0 \to I \to R  \to R' \to 0.
\end{equation}
Since $X$ is in uniform position, we have $\min\{ \, i \mid I_i \ne 0 \, \} = s \ge 2$. 
Applying $\Tor_\bullet^S(\kk,-)$ to the sequence \eqref{IRR'}, we have $\beta_{n,n+1}(I)=0$, and hence
 $\beta_{n,n+1}(R) \le \beta_{n,n+1}(R')=0$. 
Here the last equality is the induction hypothesis, since $X' \subset \PP^n$ is in uniform position again. 
 \end{proof}


\begin{prop}\label{deg B > c+1}
Let $A$ be a Cohen--Macaulay semi-standard graded ring with the $h$-vector $(h_0,h_1,h_2)$. Assume that $B=\kk[A_1]$ is a domain  
(then $\deg B \ge \codim B +1 = h_1 +1$).  If $\deg B > h_1 +1$, then $A$ is level.  
\end{prop}

\begin{proof}
First, we remark that if $A$ is not level then $\dim B\ge 2$. In fact, if $\dim B=1$, then $B$ is a polynomial ring (since $B$ is a standard graded domain now) and $\deg B =1$. This contradicts the assumption that $\deg B > h_1 +1$. 
 
We use the same notation as the previous section. So $S$ is the standard graded polynomial ring with $S_1 = A_1 (=B_1)$. 
By Bertini's theorem, if $\dim B \ge 3$, then there is some $x \in B_1 =S_1$ such that $\tilde{B}:=(B/xB)^\sat$ is a domain with $\dim \tilde{B}=\dim B-1$. 
Then $x$ is a non-zero divisor of $A$, and  $A':=A/xA$ is  a Cohen--Macaulay semi-standard graded ring with the $h$-vector $(h_0, h_1, h_2)$. Moreover,  $B':=B/(xA \cap B)$ is the subalgebra of $A'$ generated by 
its degree 1 part.  Since $H_\mm^0(B') \subset H_\mm^0(A') =0$, $B'$ is a quotient ring of $\tilde{B}$. However, $\tilde{B}$ is a domain with $\dim B'=\dim A'=\dim \tilde{B}$, hence we have 
$B'=\tilde{B}$.  In particular, $\deg \tilde{B}=\deg B$. 
Since  $A'$ is level if and only if so is $A$,  we can reduce the statement on $A$ and $B$ to that on $A'$ and $B'$. 

Repeating the above argument, we may assume that $\dim A=2$ (i.e, $\Proj B$ is a curve).  
In this case, there is some $x \in B_1 =S_1$ such that $\tilde{B}:=(B/xB)^\sat$ is  the homogeneous coordinate ring of a finite set of points 
in uniform position by Theorem~\ref{UPL}.  Let $B':= B/(xA \cap B)$ be the subalgebra of $A':=A/xA$ generated by its degree 1 part. 
Since $H_\mm^0(B') \subset H_\mm^0(A')  =0$, $B'$  is a quotient ring of $\tilde{B}$.  
Moreover, since $A_1 =B_1$, we have $[(xA \cap B)/xB]_i =0$ for all $i \le 2$, and hence $\tilde{B}_i = (B/xB)_i = B'_i$  for all $i \le 2$.

Set $S':=S/xS$ and $c:= \dim S'-1 =h_1$. For $X:=\Proj \tilde{B}$ and $Y:= \Proj B'$, we have $Y \subset X \subset \PP^c$. 
Since $\tilde{B}_i =B'_i$ for $i \le 2$, we have $\# Y \ge H_Y(2) = H_X(2)> c+1$. Since $Y$ is in uniform position in $\PP^c$, we have $\beta_{c,c+1}^{S'}(B')=0$ by Lemma~\ref{uniform level}. 

Now let us prove that $A'$ is level.    
By Lemma~\ref{level by Betti}, it suffices to show that  $\beta_{c,c+1}^{S'}(A') = 0$. Consider the exact sequence 
\begin{equation}\label{A'B'C'}
0 \to B' \to A' \to C' \to 0
\end{equation} 
of $S'$-modules.  Since $C'_i=0$ for $i \le 1$,  we have $\beta_{i,j}^{S'}(C')=0$ for all $i,j$ with $j \le i+1$. 
Applying $\Tor_\bullet^{S'}(\kk,-)$ to \eqref{A'B'C'}, it follows that 
$\beta_{i,i+1}^{S'}(A')=\beta_{i,i+1}^{S'}(B')$ for all $i$. Since $\beta_{c,c+1}^{S'}(B')=0$ as we showed above, we have 
$\beta_{c,c+1}^{S'}(A')=0$. It means that $A'$ is level, and so is $A$ itself.  
 \end{proof}

\begin{theorem}\label{main1}
Let $A$ be a Cohen--Macaulay semi-standard graded ring with the $h$-vector $(h_0,h_1,h_2)$.  If $A$ is not level and $B=\kk[A_1]$ is a domain, then $h_1+1$ divides $h_2$.  
\end{theorem}

\begin{proof}
If $\dim B=1$, then $B$ is a polynomial ring and $h_1=0$. This contradicts the assumption that $A$ is not level by Lemma~\ref{h1=0}. 
Hence we have  $\dim B \ge 2$. By the argument using Bertini's theorem, 
we can reduce to the case  $\dim B = 2$ as in the proof of Proposition~\ref{deg B > c+1}. By the proposition, we have 
$\deg B = h_1 +1$, and hence $B$ is the homogeneous coordinate ring of a rational normal curve as we have remarked in the last of  the previous section. 
In other words, $B \cong \kk[x^n, x^{n-1}y, \ldots, xy^{n-1}, y^n]$, where $n+1=h_1 +2=\dim_\kk B_1$, that is, $B$ is isomorphic to the $n$th Veronese subring $T^{(n)}=\bigoplus_{i \in \NN} T_{ni}$ of the 2-dimensional polynomial ring $T= \kk[x,y]$.   
Let $S$ be the polynomial ring of $n+1$ variables, and regard $B$ as a quotient ring of $S$ as before. Then we have $\reg_S B=2$.  


Note that  $A$ is a 2-dimensional Cohen--Macaulay graded $B$-module.   Next we consider the direct sum decomposition of $A$ as a $B$-module. 
Identifying $B$ as the  $n$th Veronese subring  $T^{(n)}$ of $T= \kk[x,y]$,    
an indecomposable Cohen--Macaulay graded $B$-module of dimension 2 is 
isomorphic to 
$$V(m) := \bigoplus_{i \in \NN} T_{m + ni}$$
for some $0 \le m < n$ up to degree shift. 
This is a classical result, and  can be proved by a similar way to its ``local version" 
(\cite[Proposition~10.5]{Yo}), since $T^{(n)}$ is an invariant subring of $T$ by a cyclic group of the order $n$. 

By the argument using Hilbert functions, we see that $A \cong B \oplus C$ as $B$-modules, where $C$ is a  2-dimensional  Cohen--Macaulay module. We have $C_i=0$ for all $i \le 1$ and $\reg C=2$. In other words, $C$ has a 2-linear resolution.

For simplicity, we set the degree 2 part of $V(m)$ as a graded $B$-module (or $S$-module) to be  $T_m$. In other words, we use the convention that $V(m)$ is generated by its degree 2 part.  The Hilbert series of $V(m)$ is 
$$\sum_{i \in \NN} (\dim_\kk [V(m)]_i) \cdot t^i = \frac{(m+1)t^2+ (n-1-m)t^3}{(1-t)^2},$$
so $V(m)$ has a 2-linear resolution as a graded $S$-module if and inly if $m = n-1$. Hence we have $C \cong (V(n-1))^{\oplus l}$ for some $l \in \NN$, and the Hilbert series of $C$ is 
$$\dfrac{nl\cdot t^2}{(1-t)^2} =\dfrac{(h_1+1)l\cdot t^2}{(1-t)^2}.$$ 
On the other hand,  the Hilbert series of $B$ is  $(1+h_1 t)/(1-t)^2$, hence the Hilbert series of $A \, (\cong B \oplus C)$ is 
$$\dfrac{1+h_1 t+ (h_1+1)l\cdot t^2}{(1-t)^2},$$ 
and the $h$-vector of $A$ is $(1, h_1, (h_1 +1)l)$. 
\end{proof}

\begin{prop}
Let $A$ be a Cohen--Macaulay semi-standard graded ring with the $h$-vector $(h_0,h_1,h_2)$.  If $A$ is not level and $B$ is a domain, then the Cohen-Macaulay type $r(A)$ of $A$ is equal to $h_1+h_2 \, (=\deg A -1)$. 
\end{prop}

Recall that $\deg A-1$ is the largest possible value of the Cohen-Macaulay type of $A$ in general. 

\begin{proof}
Since $A$ is not level, we have $c:=h_1 >0$ and we can use Lemma~\ref{CM type}. With the same notation as in the proof of Theorem~\ref{main1}, we have
$$r(A)=\beta_c(A)=\beta_c(B)+\beta_c(C).$$
However, easy calculation shows  that $\beta_c(B)=h_1$ and $\beta_c(C)=h_2$. So we are done. 
\end{proof}

\section{Semi-standard graded normal affine semigroup rings and non-level Ehrhart rings}

In this section, we recall some notions for affine semigroup rings (toric rings) 
and we discuss semi-standard graded normal affine semigroup rings. 
We also introduce Ehrhart rings which are some kind of normal affine semigroup rings arising from lattice polytopes. 
It will be proved that every semi-standard graded normal affine semigroup ring can be always 
viewed as the Ehrhart ring of some lattice polytope. 
Finally, we will show the examples of non-level Ehrhart rings 
whose $h$-vectors are of the form $(1,h,n(h+1))$ (Theorem \ref{thm:Ehr}).

For $A \subset \RR^d$, let $\RR_{\geq 0}A$ denote the cone generated by $A$, i.e., 
$$\RR_{\geq 0}A=\left\{\sum_{v \in A} r_v v \in \RR^d : r_v \in \RR_{\geq 0}\right\}.$$ 
For $B \subset \ZZ^d$, let $\gp(B)$ denote the group (the lattice) generated by $B$, i.e., 
$$\gp(B)=\left\{\sum_{v \in B} z_v v \in \ZZ^d : z_v \in \ZZ\right\}.$$

Let $C \subset \ZZ^d$ be an affine semigroup. 
\begin{itemize}
\item We say that $C$ is {\em pointed} if $C$ contains no vector subspace of positive dimension. 
\item Let $\overline{C}=\RR_{\geq 0}C \cap \gp(C)$. We say that $C$ is {\em normal} if $C=\overline{C}$. 
\item Let $\kk[C]$ be the affine semigroup ring of $C$, i.e., 
$$\kk[C]:=\kk[{\bf X}^\alpha : \alpha=(\alpha_1,\ldots,\alpha_d) \in C],$$ 
where ${\bf X}^\alpha=\prod_{i=1}^d X_i^{\alpha_i}$ denotes a Laurent monomial. 
Note that $\kk[C]$ is positively graded if and only if $C$ is pointed, and $\kk[C]$ is normal if and only if $C$ is normal. 
\end{itemize}

When $\kk[C]$ is $\ZZ$-graded, by abuse of notation, 
we write $\deg(\alpha)=n$ if $\deg({\bf X}^\alpha)=n$ for $\alpha \in C$.

We also recall what the Ehrhart ring of a lattice polytope is. 
Let $P \subset \RR^d$ be a lattice polytope, which is a convex polytope all of whose vertices belong to 
the standard lattice $\ZZ^d$, of dimension $d$. We define the $\kk$-algebra $\kk[P]$ as follows: 
\begin{align*}
\kk[P]=\kk[ {\bf X}^\alpha Z^n : \alpha \in nP \cap \ZZ^d, \; n \in \NN], 
\end{align*}
where for $\alpha=(\alpha_1,\ldots,\alpha_d) \in \ZZ^d$, ${\bf X}^\alpha Z^n=X_1^{\alpha_1} \cdots X_d^{\alpha_d} Z^n$ 
denotes a Laurent monomial in $\kk[X_1^\pm, \ldots,X_d^\pm, Z]$ and $nP=\{nv : v \in P\}$. 
It is known that $\kk[P]$ is a semi-standard graded normal Cohen--Macaulay domain of dimension $d+1$, 
where the grading is defined by $\deg ({\bf X}^\alpha Z^n) =n$ for $\alpha \in nP \cap \ZZ^d$. 
The graded $\kk$-algebra $\kk[P]$ is called the {\em Ehrhart ring} of $P$. 

\begin{rem}\label{chuui}
Let $C \subset \ZZ^d$ be a pointed normal affine semigroup and assume that $\kk[C]$ is $\ZZ$-graded. 
Let $C_1$ be the set of all degree one elements of $C$. 
If $C$ satisfies $C=\RR_{\geq 0}C_1 \cap \gp(C)$, then there exists a lattice polytope $P$ 
such that $\kk[C]$ is isomorphic to the Ehrhart ring of $P$ as $\kk$-algebras. 
In fact, let $Q$ be the convex hull of $C_1$. Then we can easily see that $\kk[C] \cong \kk[Q]$. 
\end{rem}

Actually, semi-standard graded normal affine semigroup rings are isomorphic to Ehrhart rings of lattice polytopes. 
\begin{prop}\label{key}
Let $C$ be a normal affine semigroup. Assume that $\kk[C]$ is semi-standard graded. 
Then there exists a lattice polytope $P$ such that $\kk[C]$ is isomorphic to the Ehrhart ring of $P$ as $\kk$-algebras. 
\end{prop}
\begin{proof}
We see that $C$ is pointed since $\kk[C]$ is positively graded. 
Let $C_1$ be the set of all degree one elements of $C$. 
The normality of $C$ implies that $C$ contains all lattice points contained in $\RR_{\geq 0}C_1$, 
where ``lattice points'' stand for the points in $\gp(C)$. Thus, $\RR_{\geq 0}C_1 \cap \gp(C) \subset \overline{C} = C$. 
It suffices to show the equality $C=\RR_{\geq 0}C_1 \cap \gp(C)$ (see Remark \ref{chuui}), 
in particular, $\RR_{\geq 0}C \subset \RR_{\geq 0}C_1$.  

To prove this, it is enough to prove that all $1$-dimensional cones in $\RR_{\geq 0}C$ lie in $\RR_{\geq 0}C_1$. 
If this were wrong, then there is a lattice point $\alpha \in C \cap \gp(C)$ 
such that $\RR_{\geq 0}\{\alpha\} \subset \RR_{\geq 0}C \setminus \RR_{\geq 0}C_1$. 
Then $m\alpha \in (\RR_{\geq 0}C \setminus \RR_{\geq 0}C_1) \cap \gp(C)$ for any $m \in \ZZ_{>0}$. 
This implies that for any positive integer $m$, ${\bf X}^{m\alpha}$ will be a generator of $\kk[C]$ as a $\kk[C_1]$-module, 
where $\kk[C_1]$ denotes a subalgebra of $\kk[C]$ generated by its degree one elements. 
This contradicts the hypothesis that $\kk[C]$ is finitely generated as $\kk[C_1]$-module, i.e., $\kk[C]$ is semi-standard. 

Therefore, $C=\RR_{\geq 0}C_1 \cap \gp(C)$. This says that $\kk[C]$ is isomorphic to the Ehrhart ring of some lattice polytope. 
\end{proof}

In the context of enumerative combinatorics on lattice polytopes, 
the $h$-vector of the Ehrhart ring $\kk[P]$ of $P$ is often called the {\em $h^*$-vector} (or the {\em $\delta$-vector}) of $P$. 
It is known that the $a$-invariant of the Ehrhart ring of $P$ can be computed as follows: 
\begin{align*}\label{socle}
a(\kk[P])=-\min\{ \ell \in \ZZ_{> 0} : \ell P^\circ \cap \ZZ^d \not= \emptyset\}, 
\end{align*}
where $P^\circ$ denotes the interior of $P$. Note that $s=d+1+a(\kk[P])$ holds 
when the $h^*$-vector of $P$ is $(h_0^*,h_1^*, \ldots,h_s^*)$. 

We can discuss whether $\kk[P]$ is level in terms of $P$ as follows. 
\begin{prop}\label{prop:criterion}
Let $P \subset \RR^d$ be a lattice polytope of dimension $d$. 
Then $\kk[P]$ is level if and only if for each $n \geq -a(\kk[P])$ and for each $\alpha \in nP^\circ \cap \ZZ^d$, 
there exist $\alpha_1,\ldots,\alpha_{n+a(\kk[P])} \in P \cap \ZZ^d$ and $\beta \in (-a(\kk[P]))P^\circ \cap \ZZ^d$ 
such that $\alpha=\alpha_1+\cdots+\alpha_{n+a(\kk[P])}+\beta$. 
\end{prop}

We recall the well-known combinatorial technique how to compute the $h^*$-vector of a lattice {\em simplex}. 
Given a lattice simplex $\Delta \subset \RR^d$ of dimension $d$ with the vertices $v_0, v_1, \ldots, v_d \in \ZZ^d$, we set 
$$S_\Delta=\left\{ \sum_{i=0}^d r_iv_i \in \ZZ^d : \sum_{i=0}^d r_i \in \NN, \; 0 \leq r_i < 1 \right\}. $$
We define $\height(\alpha)=\sum_{i=0}^d r_i$ for each $\alpha=\sum_{i=0}^dr_iv_i \in S_\Delta$. 

\begin{lem}[cf. {\cite[Proposition 27.7]{HibiRedBook}}]\label{compute}
Let $(h_0^*,h_1^*,\ldots,h_s^*)$ be the $h^*$-vector of $\Delta$. 
Then one has $s=\max\{ \height(\alpha) : \alpha \in S_\Delta\}$ and 
$$h_i^*=|\{ \alpha \in S_\Delta : \height(\alpha)=i\}|$$
for each $i=0,1,\ldots,s$. Moreover, $\sum_{i=0}^d h_i^* = |S_\Delta|=\text{(the volume of $\Delta$)} \cdot d!$. 
\end{lem}

The following is the main result of this section. 
\begin{theorem}\label{thm:Ehr}
Given positive integers $h$ and $n$, there exists a lattice polytope $P_{h,n}$
such that its Ehrhart ring $\kk[P_{h,n}]$ satisfies the following: 
\begin{itemize}
\item the $h$-vector of $\kk[P_{h,n}]$ (the $h^*$-vector of $P_{h,n}$) is $(1,h,n(h+1))$; 
\item $\kk[P_{h,n}]$ is non-level. 
\end{itemize}
\end{theorem}
\begin{proof}
Let $v_0=(1,1,n), v_1=(0,1,0),v_2=(0,0,1)$ and $v_3=(1,-h,-nh)$ and let $P_{h,n}$ be the convex hull of them. 
Then $P_{h,n}$ is a lattice simplex of dimension $3$ with its vertices $v_0,v_1,v_2,v_3$. 
We will prove that $\kk[P_{h,n}]$ satisfies the required properties. 

Let $\Delta=P_{h,n}$. First of all, we see that $\vol(\Delta) \cdot 3! =(n+1)(h+1)$ 
by calculating the determinant of the matrix $(v_1-v_0, v_2-v_0, v_3-v_0)$. 

Next, let us compute $S_\Delta$. Let $$v:=\frac{h}{h+1}v_0+\frac{1}{h+1}v_3=(1,0,0) \in \Delta \cap \ZZ^3.$$ 
Then $v \in S_\Delta$ with $\height(v)=1$. 

\begin{itemize}
\item For each $i=1,2,\ldots,h-1$, let 
$$w_i:=\frac{i}{h}v_3+\frac{h-i}{h}v=(1,-i,-ih).$$ Then $w_i \in S_\Delta$ with $\height(w_i)=1$. 
\item For each $j=1,2,\ldots,n$, let 
$$w_j':=\frac{j}{n+1}v_0+\frac{n+1-j}{n+1}v_1+\frac{j}{n+1}v_2+\frac{n+1-j}{n+1}v=(1,1,j).$$
Then $w_j' \in S_\Delta$ with $\height(w_j')=2$. 
\item For each $q=0,1,\ldots,h-1$ and $r=1,2,\ldots,n$, let 
\begin{align*}
u_{q,r}:&=\frac{rh}{(n+1)h}v_1+\frac{n+1-r}{n+1}v_2+\frac{(n+1)q+r}{(n+1)h}v_3+\frac{(n+1)(h-q)-r}{(n+1)h}v \\
&=(1,-q,1-nq-r). 
\end{align*}
Then $u_{q,r} \in S_\Delta$ with $\height(u_{q,r})=2$. 
\end{itemize}
Hence, we obtain that 
\begin{align*}S_\Delta&=\{(0,0,0)\}\cup\{v\} \cup \{w_i : 1 \leq i \leq h-1\} \\
&\cup \{w_j' : 1 \leq j \leq n\} \cup \{u_{q,r} : 0 \leq q \leq h-1, \; 1 \leq r \leq n\}.\end{align*}
Thus, the $h^*$-vector of $\Delta$ coincides with $(1,h,n(h+1))$ by Lemma \ref{compute}. 

In addition, we also see that 
\begin{align*}
\Delta \cap \ZZ^3 &= \{v_0,v_1,v_2,v_3\} \cup \{v\} \cup \{w_i : 1 \leq i \leq h-1\}, \text{ and }\\
2\Delta^\circ \cap \ZZ^3 &= \{w_j':1 \leq j \leq n\} \cup \{u_{q,r}: 0 \leq q \leq h-1, 1 \leq r \leq n\}.
\end{align*}
Consider $$w:=v_1+v_2+v=(1,1,1) \in 3\Delta^\circ \cap \ZZ^3.$$ 
Then we can see that there are no $\alpha \in \Delta \cap \ZZ^3$ and $\beta \in 2\Delta^\circ \cap \ZZ^3$ such that $w=\alpha+\beta$. 
This implies that $\kk[\Delta]$ is non-level by Proposition \ref{prop:criterion}. 
\end{proof}

\begin{rem}\label{degree 2 Ehrhart}
It is known by \cite{HT} and \cite{Treut} that for integers $a \ge 0$ and $b >0$, 
$(1,a,b)$ is the $h^*$-vector of some lattice polytope, 
i.e., the $h$-vector of some semi-standard graded normal affine semigroup ring 
if and only if $a \leq 3b+3$ or $(a,b)=(7,1)$ holds. 
The ``If'' part was proved in \cite{HT} and the ``Only if'' part was proved in \cite{Treut}. 
In \cite{HT}, for  $a \ge 0$ and $b >0$  satisfying $a \leq 3b+3$ or $(a,b)=(7,1)$, 
the lattice polytope whose $h^*$-vector coincides with $(1,a,b)$ is given. 
We can see that the associated Ehrhart rings of their examples are all {\em level}. 
Namely, there exists a level Ehrhart ring with the $h$-vector $(1,h,n(h+1))$ for any positive integers $h$ and $n$. 
\end{rem}

\begin{rem}
Let $C$ be a smooth projective curve of genus $g \, (\ge 1)$,  
and   $\cL$ an invertible sheaf on $C$ with $\deg \cL = 2g+c$ for some $c \ge 1$. Then, by \cite[Corollary to Theorem 6]{Mum}, the ring  
$$A = \bigoplus_{ n \in \NN} H^0(C, n \cL)$$
is a normal Cohen--Macaulay standard graded domain with the $h$-vector 
$(1, g+c-1,g)$. 
(If $c$ is non-positive but not so small, then $A$ has the same property in many cases. This is a classical topic of the curve theory, but we do not argue this direction here.)

Anyway, combining this observation with Remark~\ref{degree 2 Ehrhart}, we see that, for any sequence $(1, a,b)$ for integers $a \ge 0$ and $b >0$, there  is a  normal Cohen--Macaulay semi-standard graded domain with the $h$-vector $(1, a, b)$. Moreover, we can take such rings from level rings. 
\end{rem}

\section*{Acknowledgements}
We are grateful to Professors Chikashi Miyazaki and Kazuma Shimomoto for stimulating discussion. We also thank Professors Takesi Kawasaki and Kazunori Matsuda for  valuable  comments which helped to improve the exposition of the paper.

\end{document}